\documentclass[11pt,letter,reqno]{amsart}
\usepackage{setspace}

\usepackage{amssymb,amsmath,amsfonts,amsthm,commath}
\usepackage{graphicx}
\usepackage{float}
\usepackage{framed}
\usepackage{calc,extarrows}
\usepackage{xcolor}

\usepackage{amsfonts}
\usepackage{amssymb}
\usepackage{graphicx}
\usepackage{newlfont}
\usepackage{mathrsfs}
\usepackage{graphics}
\usepackage{amsmath, amssymb, amsfonts,amsthm}
\usepackage{textcomp}
\usepackage{mathtools}
\usepackage{multicol}
\usepackage{bbm}
\usepackage{float}
\usepackage{enumerate}

\usepackage{hyperref}

\usepackage[numbers,sort&compress]{natbib}
\usepackage{marginnote}

\newtheorem{theorem}{Theorem}[section]
\newtheorem{corollary}[theorem]{Corollary}

\newtheorem{proposition}[theorem]{Proposition}

\theoremstyle{remark}

\newcommand{\cd}{\ \stackrel{d}{\longrightarrow} \ }

\newcommand{\cas}{\ \stackrel{\mbox{a.s.}}{\longrightarrow} \ }

\newcommand{\GG}{\mbox{${\mathcal G}$}}

\newcommand{\Var}{\mathbf{Var}}

\newcommand{\sfrac}[2]{{\textstyle\frac{#1}{#2}}}

\newcommand{\bone}{\mathbf{1}}

\numberwithin{equation}{section}

\newtheorem{lemma}[theorem]{Lemma}

\DeclarePairedDelimiter\floor{\lfloor}{\rfloor}

\begin{document}

\title{De-Preferential Attachment Random Graphs}
\author{Antar Bandyopadhyay} 
\address[Antar Bandyopadhyay]{Theoretical Statistics and Mathematics Unit \\
         Indian Statistical Institute, Delhi Centre \\ 
         7 S. J. S. Sansanwal Marg \\
         New Delhi 110016 \\
         INDIA}
\address{Theoretical Statistics and Mathematics Unit, 
         Indian Statistical Institute, Kolkata;
         203 B. T. Road, Kolkata 700108, INDIA}
\email{antar@isid.ac.in}          
\author{Subhabrata Sen}  
\address[Subhabrata Sen]{Department of Statistics\\
         Harvard University\\
         1 Oxford Street, SC 713\\
         Cambridge, MA 02138\\
         USA}
\email{subhabratasen@harvard.edu}

\begin{abstract}
\noindent
In this work we consider a growing random graph sequence where a new vertex is less likely to join 
to an existing vertex with high degree and more likely
to join to a vertex with low degree. In contrast to the well studied 
\emph{preferential attachment random graphs} \cite{BarAlb99}, we
call such a sequence a 
\emph{de-preferential attachment random graph model}. We consider two types of models, 
namely, \emph{inverse de-preferential}, where
the attachment probabilities are inversely proportional to the degree and 
\emph{linear de-preferential}, where the attachment probabilities are
proportional to $c-$degree, where $c > 0$ is a constant.  
For the case when each new vertex comes with 
exactly one half-edge we show that the degree of a fixed vertex is asymptotically of the order 
$\sqrt{\log n}$ for the inverse de-preferential case and
of the order $\log n$ for the linear case. These show that compared to preferential attachment,  
the degree of a fixed vertex grows to infinity at a much slower rate for these models. 
We also show that in both cases limiting degree distributions have exponential tails.
In fact we show that for the inverse de-preferential model the tail of the limiting degree distribution 
is faster than exponential while that for the linear de-preferential model is exactly the 
$\mbox{Geometric}\left(\frac{1}{2}\right)$ distribution.
For the case when each new vertex comes with $m > 1$ half-edges, we show that similar 
asymptotic results hold for fixed vertex degree in both inverse and linear de-preferential models. 
Our proofs make use of the martingale approach as well as embedding to certain continuous time age dependent branching processes. 
\end{abstract}



\maketitle

\section{Introduction}
Networks are ubiquitous in our surroundings. Complex biological networks such as  protein-interaction networks, social networks, and electronic networks 
(such as the HTTP network of the WWW) form the very basis of modern human existence. Across disciplines, many scientists have attempted to study the 
properties of these complex structures. With increased computational power, it has become possible to study large real-life networks. These empirical 
studies have observed some distinctive properties (such as ``scale-free" structure and ``small-world" property) which are exhibited by a number of 
complex networks. 

Random graph models have been put forward to explain specific properties observed in complex networks. We will be specifically interested in the 
distribution of vertex degrees in large networks. Empirical studies have reported that many complex networks have degree distributions which are 
``scale-free", that is, the tail of the degree distribution decays roughly like a power-law.

The most classical model for a random graph is the Erd\'{o}s- R\'{e}nyi random graph \cite{JanLucRuc2000}. 
However, asymptotically, the tail of the vertex degree distribution in an Erd\'{o}s- R\'{e}nyi random graph decays exponentially. 
Thus these random graphs are not suitable models for real-life complex networks. This has necessitated the development of other models which would 
exhibit these unique properties observed in real networks. 

Preferential attachment random graphs, such as the Albert-Barabasi Model \cite{BarAlb99}, have emerged as a popular choice for modeling complex networks. 
In this class of models, one starts with a simple initial configuration of vertices (e.g. two vertices joined by an edge) and adds a new vertex at each 
step. In the simplest case, an existing vertex is chosen with probability proportional to some weight $w:\mathbb{N} \to \mathbb{R}^+$ of the degree of 
the vertex, and the new vertex is joined by an edge to the sampled vertex. The function $w$ is taken to be an increasing function --- 
high degree vertices are therefore more likely to be attached to the new vertex. This leads to the term ``preferential attachment". 
If $w$ is taken to be linear, the degree distribution of the resulting random graph actually exhibits the ``scale-free" property 
(see e.g. \cite{Hofs2012}).
 
In a general version of the preferential attachment model, one similarly starts with a simple configuration of vertices and adds a vertex at each step. 
However, in this case, the new vertex is attached to a fixed number $m$ of existing vertices. Often, it will be convenient to imagine that the new vertex 
has $m$- half edges attached to it, and these half-edges are attached to existing vertices. Numerous alternative models have been suggested in the 
literature for choosing $m$ vertices. For example, the vertices may be chosen in an independent and identically distributed manner with probabilities 
proportional to a function of their degrees, or the $m$ half edges attached to the new vertex may be sequentially joined to existing vertices 
(sampled with probabilities proportional to a function of their degrees) and the degrees of the vertices might be updated during the intermediate steps. 
 

Here we explore the opposite phenomenon. 
We start with a simple configuration of vertices and add a new vertex at each step. 
We study two random graph models where vertices with high degrees are less likely to be attached to the new vertex. These random graphs will be called 
``De-preferential Attachment" random graphs. Having initiated the study independently, we later discovered that such models have been discussed in Physics literature
earlier as models for food webs \cite{SevRik2006, SevRik2008}. It is interesting to note that these papers contain non-rigorous study of the model which 
leads to the same conclusions. However, it must also be stated that the papers include more results than what we have been able to
derive rigorously. It is perhaps worthwhile to continue the study to establish all the results which the physicists have derived using their
non-rigorous arguments. Similar models have also been studied mainly through simulations 
in Computer Science literature for modeling ``peer-to-peer'' networks \cite{MirFig2009}.
In the context of certain other type of random reinforced models, such as, 
urn models, the concept of ``De-preferential" selection has been studied in the recent literature \cite{BaKa2018, Ka2019}.

We will be specifically interested in the evolution of the degree of a fixed vertex and the asymtotic degree distribution of these two random graph models. 
Our results for one of the models will be critically dependent on the embedding of the discrete time random graph process into a sequence of continuous time pure 
birth processes.
\subsection{Model}
The two de-preferential attachment random graph models will be denoted by ``linear de-preferential attachment model" and 
``inverse de-preferential attachment model". For both models, we start with a simple configuration of vertices and add a single new 
vertex at each step. Initially, we allow the new vertex to join to one of the existing vertices. Later, we generalize both models and allow the new 
vertex to join to $m (\geq 1)$  existing vertices. It is worth noting that if $m > 1$ then it is possible to get \emph{loops} at a single vertex
or \emph{multiple edges} between two vertices. 
As in the case of preferential attachment random graphs, numerous alternatives may be suggested for 
choosing the $m$ vertices which are joined to the new vertex. 

In both the models, each of the $m$-half edges of a vertex is attached to one of the existing vertices. The attachment is carried out sequentially and the degrees of the existing vertices are updated during the intermediate steps. We will see that this choice allows us to couple the inverse de-preferential attachment model naturally with a continuous time age-dependent branching process for $m=1$ case and with a sequence of pure birth processes for $m>1$. This coupling will play a critical role in the analysis of the inverse de-preferential attachment model. 

Initially, our starting configuration is a ``hanging" tree consisting of two vertices joined by 
an edge and a ``free" edge linked to one of the vertices. We add vertices sequentially and attach the new vertices randomly to the existing vertices to form a sequence of growing graphs. Later, we generalize the models to include the $m>1$ scenario. In this case, our starting configuration
is a graph consisting of two vertices joined by $m$ edges and $m$ free edges linked to one of the vertices. 
\subsection{Notation} 
We introduce some notation that will be used throughout this paper. We denote the random graph process 
by $\{G_n\}_{n=2}^{\infty}$. The graph $G_n$ has
vertices $V(G_n)= \{v_1,\cdots, v_n\}$. At time $(n+1)$, we introduce the vertex $v_{n+1}$ with 
half edges $e_{n+1,1},\cdots,e_{n+1,m}$. Also, let $d_i(n+1,k)$, $k=0,\cdots,m-1$, denote the degree of the
vertex $v_i$, $i=1,\cdots,n$, after $k$ half-edges of $v_{n+1}$ have been attached to the graph. 
Here degree refers to both the \emph{in} and \emph{out}-degree of a vertex.
We will use $\{\mathscr{F}_{n,k}: n \geq 2, k=0,\cdots, m-1 \}$ to denote the natural filtration associated
with the random graph process. Finally, let $d_i(n+1,0) = d_i(n)$. For $m=1$, the natural filtration will be 
simply denoted by $\{\mathscr{F}_n:n \geq 2\}$ and the half edge of vertex $v_{n+1}$ will be denoted by $e_{n+1}$. 

Finally, let $N_k(n)$ denote the number of vertices of degree $k$ in $G_n$. $P_k(n) = \frac{N_k(n)}{n}$ denotes the empirical proportion of vertices of degree $k$ in $G_n$.

\subsubsection{Linear De-preferential}
For $m=1$, 
\begin{align*}
P(e_{n+1}=\{v_{n+1}, v_i\} | \mathscr{F}_n) &\propto \left( 1- \frac{d_i(n)}{2n-1}\right)\\
\implies P(e_{n+1}=\{v_{n+1}, v_i\} | \mathscr{F}_n) &= \frac{1}{n-1} \left( 1- \frac{d_i(n)}{2n-1}\right)\\
\end{align*}
For $m>1$, $j=1,\cdots,n$, and $k= 0,1,\cdots,m-1$,
\begin{align*}
P(e_{n+1,k+1}=\{v_j,v_{n+1}\}| \mathscr{F}_{n+1,k}) &= \frac{1}{n-1}\left(1- \frac{d_j(n+1,k)}{ k+(2n-1)m}\right)
\end{align*}

\subsubsection{Inverse De-preferential}
For $m=1$, 
\begin{align*}
 P(e_{n+1}= \{v_{n+1}, v_i\} |\mathscr{F}_{n}) &\propto \frac{1}{d_i(n)} \\
\implies P(e_{n+1}= \{v_{n+1}, v_i\} |\mathscr{F}_{n}) &= C_n \frac{1}{d_i(n)}
\end{align*}

where $C_n ^{-1} = D_n= \displaystyle\sum_{i=1} ^{n} \frac{1}{d_i(n)}$.

For $m>1$, $j=1,\cdots,n$, $k=0,\cdots,m-1$,
\begin{align*}
P(e_{n+1,k+1}=\{v_j,v_{n+1}\}| \mathscr{F}_{n+1,k}) &\propto \frac{1}{d_j(n+1,k)}\\
\implies P(e_{n+1,k+1}=\{v_j,v_{n+1}\}| \mathscr{F}_{n+1,k}) &= C_{n+1,k}\frac{1}{d_j(n+1,k)}
\end{align*}
where $C_{n+1,k}^{-1} = D_{n+1,k}= \displaystyle\sum_{j=1}^{n} \frac{1}{d_j(n+1,k)}$.

\subsection{Outline}
The rest of the paper is organized as follows. In Section~\ref{results} 
we state our main results
for both the linear and inverse de-preferential models, while the proofs are presented in Sections 
~\ref{proofs-lin} and ~\ref{proofs-inv} respectively.

\section{Main Results}
In this Section we state the main results for the two random graph models. For both models, we start with the 
results for $m=1$ case and then we will state the results for the $m>1$ case. 

\label{results}
\subsection{Linear De-preferential}
\subsubsection{$m=1$ Case}
We now state the results for the $m=1$ case.
\begin{theorem}
\label{ModelA_rate}
Let $\left(G_n\right)_{n=1}^{\infty}$ be a sequence of 
random graphs following a 
\emph{linear de-preferential attachment} model with
$m=1$. Then as $n \rightarrow \infty$,
\begin{equation}
\frac{d_i(n)}{\log n} {\buildrel P \over \longrightarrow} 1.
\end{equation}
\end{theorem}

Thus the degree of a fixed vertex grows like $\log n$. In Athreya et. al.(\cite{AthGhoSeth2008}), similar
rates have been derived for linear preferential attachment and sublinear preferential attachment models
when $w(k)= k^p, \frac{1}{2}<p<1$. In comparison, the
degree of a fixed vertex grows like $O(\sqrt{n})$ in linear preferential attachment, and 
like $O((\log n)^q)$, where $q=\frac{1}{1-p}$ when $w(k)= k^p , \frac{1}{2}<p<1$. Hence, the degree
of a fixed vertex grows at a slower rate in comparison to linear and this class of sub-linear preferential attachment models. 

We also have a central limit theorem for this model.
\begin{theorem}
\label{ModelA_central}
Let $\left(G_n\right)_{n=1}^{\infty}$ be a sequence of 
random graphs following a 
\emph{linear de-preferential attachment} model with
$m=1$. Then as $n \rightarrow \infty$,
\begin{equation}
\frac{d_i(n) - \log n}{\sqrt{\log n}} \cd N(0,1).
\end{equation}
\end{theorem}

The next theorem gives the limit of the empirical distribution of vertex degrees.
\begin{theorem}
\label{ModelA_degree}
Let $\left(G_n\right)_{n=1}^{\infty}$ be a sequence of 
random graphs following a 
\emph{linear de-preferential attachment} model with
$m=1$. 
Let 
\begin{equation}
P_k(n) := \frac{1}{n} \sum_{i=1}^n \bone\left(d_i(n) = k\right),
\end{equation}
be the fraction of vertices with degree $k$. 
Then $\exists\,\, C_1 > 0$, such that, as $n \rightarrow \infty$
\begin{equation}
P\left(\displaystyle\max_k \left\vert P_k(n) - p_k \right\vert > C_1 \left(\frac{\log n}{n} + \sqrt{\frac{\log n}{n}}\right) \right)= o(1),
\end{equation}
where $p_k = \frac{1}{2^k}$, $k\geq 1$.
\end{theorem}

We note that the limit $\{p_k\}_{k=1}^{\infty}$ decays exponentially. 
It is interesting to note that R\'{e}nyi had identified the same limit for the empirical degree distribution of a random graph process where new vertices are added sequentially and attached uniformly to one of the existing vertices. We interpret this as follows: asymptotically, the degrees $d_i(n)$ are negligible in comparison to $2(n-1)$ and therefore, asymptotically, every vertex is attached ``approximately" uniformly to one of the existing vertices. In this sense, the linear de-preferential attachment model represents a weak form of de-preferential attachment. We will see that the inverse de-preferential attachment model represents a stronger form of de-preferential attachment. 

Finally, we have the asymptotic size biased distribution.
\begin{theorem}
\label{ModelA_sizebiased}
Let $\left(G_n\right)_{n=1}^{\infty}$ be a sequence of 
random graphs following a 
\emph{linear de-preferential attachment} model with
$m=1$. Let $v < n$ be a fixed vertex and $k \geq 1$.
Then as $n \rightarrow \infty$,
\begin{equation}
P\left(e_{n+1}=\{v_{n+1},v\}, d_v(n)=k\right) \to \frac{1}{2^k}.
\end{equation}
\end{theorem}

\subsubsection{$m > 1$ Case}
We now state the generalizations of the above results for $m > 1$.

\begin{theorem}
\label{modelA_m}
Let $\left(G_n\right)_{n=1}^{\infty}$ be a sequence of 
random graphs following a 
\emph{linear de-preferential attachment} model with
$m > 1$ and $d_i(n)$ be the degree of a fixed vertex 
$i \geq 1$. Then as $n \rightarrow \infty$,
\begin{equation}
\frac{d_i(n)}{m \log n} {\buildrel P \over \longrightarrow} 1.
\end{equation}
\end{theorem}

The next result collects the CLT for the $m>1$ case. 

\begin{theorem}
\label{ModelA_m_clt}
Let $\left(G_n\right)_{n=1}^{\infty}$ be a sequence of 
random graphs following a 
\emph{linear de-preferential attachment} model with
$m>1$. Then as $n \rightarrow \infty$,
\begin{equation}
\frac{d_i(n) - m \log n}{\sqrt{m \log n}} \cd N(0,1).
\end{equation}
\end{theorem}

\noindent
{\bf Remark:}
Unfortunately, we have been unable to derive an analogue of Theorem ~\ref{ModelA_degree} for $m > 1$. 
Needless to say that we expect the empirical degree distribution
to converge just like in the $m=1$ case, but have been
unsuccessful in conjecturing the form of the limiting degree distribution; this has been the main barrier in establishing such a result. 

\subsection{Inverse De-preferential}
\subsubsection{$m=1$ Case}
We have the following results for the $m=1$ case. 
Let $\lambda^* >0$ satisfy the equation 
\begin{equation}
\sum_{n=1}^{\infty} \prod_{i=1}^{n} \frac{1}{1+i\lambda^*}=1.
\end{equation}

\begin{theorem}
\label{ModelB_exact}
Let $\left(G_n\right)_{n=1}^{\infty}$ be a sequence of 
random graphs following a 
\emph{inverse de-preferential attachment} model with
$m=1$. Then as $n \rightarrow \infty$,
\begin{equation}
\frac{d_i(n)}{\sqrt{\log n}} \to \sqrt{\frac{2}{\lambda^*}}
\mbox{\ \ a.s.}
\end{equation}
\end{theorem}

Our next results identify the limit of the empirical degree distribution and its properties.

\begin{theorem}
\label{ModelB_degree}
Let $\left(G_n\right)_{n=1}^{\infty}$ be a sequence of 
random graphs following a 
\emph{inverse de-preferential attachment} model with
$m=1$. 
Let 
\begin{equation}
P_k(n) := \frac{1}{n} \sum_{i=1}^n \bone\left(d_i(n) = k\right),
\end{equation}
be the fraction of vertices with degree $k$. 
Then 
$\forall k \geq 1$, as $n \to \infty$, 
\begin{equation}
P_k(n) \to \frac{k\lambda^*}{ k\lambda^* + 1} \displaystyle\prod_{i=1}^{k-1} \frac{1}{i\lambda^* + 1}
\mbox{\ \ \ \ a.s.}
\end{equation}
\end{theorem}

The proof of this result will critically exploit the embedding of this discrete time random graph process into a continuous time age dependent branching process where each vertex reproduces according to i.i.d copies of a particular pure birth process. $\lambda^*$ is actually the Malthusian Parameter of this pure birth process. 

\begin{corollary}
\label{ModelB_character}
The limiting empirical degree distribution of  
a sequence of 
random graphs following the 
\emph{inverse de-preferential attachment} model with
$m=1$
has mean $2$, mode $1$ and its tail decays at a rate faster than exponential. 
\end{corollary}

Finally, in this case, we also have the asymptotic size biased distribution for the de-preferential attachment model. 

\begin{theorem}
\label{ModelB_sizebiased}
Let $\left(G_n\right)_{n=1}^{\infty}$ be a sequence of 
random graphs following the 
\emph{inverse de-preferential attachment} model with
$m=1$. Let $v < n$ be a fixed vertex and $k \geq 1$.
Then as $n \rightarrow \infty$,
\begin{equation}
P\left(e_{n+1}=\{v_{n+1},v\}, d_v(n)=k\right) \to 
\displaystyle\prod_{i=1}^{k} \frac{1}{1+i\lambda^*}.
\end{equation}
\end{theorem}

\subsubsection{$m > 1$ Case}
We only have the following result for the $m>1$ case in the 
\emph{inverse de-preferential} model. 
\begin{theorem}
\label{ModelB_m}
Let $\left(G_n\right)_{n=1}^{\infty}$ be a sequence of 
random graphs following the 
\emph{inverse de-preferential attachment} model with
$m > 1$. Then 
$\exists$ constants $0<c<C$, such that,
as $n \rightarrow \infty$
\begin{equation}
P\left(c \leq \frac{d_i(n)}{m\sqrt{\log n}}\leq C\right) \to 1.    
\end{equation}
\end{theorem}

This last result is unsatisfactory, as it only states that 
asymptotically for every fixed vertex $i \geq 1$ the sequence 
of random variables 
\[\left(\frac{d_i(n)}{m\sqrt{\log n}}\right)_{n \geq 1}\]
remains \emph{tight}. We believe that they must be converging 
at least in
probability (and equivalently weakly) to a limiting constant.

\section{Proofs for the Linear De-preferential Models}
\label{proofs-lin}

This section includes the proofs of the main results for the linear de-preferential case. We rely on martingale techniques to derive our results. 


We start by presenting the proofs for Theorem \ref{modelA_m} and Theorem \ref{ModelA_m_clt}. Theorem \ref{ModelA_rate} and Theorem \ref{ModelA_central} follow by putting $m=1$.

\subsection{Proof of Theorem \ref{modelA_m}}
 We derive a recursive relation for $E[d_i(n)]$.
\begin{align*}
E[d_i(n+1)] &= E[d_i(n+1,m-1)] + \frac{1}{n-1} \left( 1- \frac{E[d_i(n+1,m-1)]}{(m-1) + m(2n-1)}\right)\\
&= \left(1- \frac{1}{n-1} \frac{1}{(m-1)+m(2n-1)}\right) E[d_i(n+1,m-1)] \\
&\ \qquad + \frac{1}{n-1}\\
&= \alpha_n E[d_i(n,m)] + \frac{\beta_n}{n-1}
\end{align*}
where $$\alpha_n= \displaystyle\prod_{j=0}^{m-1}\left(1- \frac{1}{(n-1)(j+ m(2n-1))}\right)$$ and 
$$\beta_n= 1+ \displaystyle\sum_{k=1}^{m-1} \displaystyle\prod_{j=m-k}^{m-1} \left(1- \frac{1}{(n-1)(j+ m(2n-1))}\right)$$. Let $a_n =E[d_i(n)], n \geq i$, $a_i = m$ w.p. 1. We define, for $n \geq i+1$
$$ \gamma_n = \prod_{k=i}^{n-1} \alpha_k$$ 
We define $\gamma_i =1$. Therefore, we have the recursion,
\begin{align*}
\frac{a_{n+1}}{\gamma_{n+1}} &= \frac{a_n}{\gamma_n} + \frac{\beta_n}{(n-1) \gamma_n}\\
\implies c_{n+1} &= c_n + \frac{\beta_n}{(n-1) \gamma_n} 
\end{align*}
where $c_n = \frac{a_n}{\gamma_n}$, $n \geq i$. Therefore, we have,
\begin{align*}
c_{n+1} &= m+ \frac{\beta_i}{(i+1)\gamma_i}+ \cdots + \frac{\beta_n}{(n+1)\gamma_n}
\end{align*}
We note that $ \beta_n \uparrow m$ as $n \to \infty$ and that $$\gamma_n \downarrow \kappa=  \displaystyle\prod_{j=0}^{m-1} \prod_{n=i}^{\infty} \left(1- \frac{1}{(n-1) (j+ m(2n-1))}\right)$$ as $n \to \infty$.
Let $h_n = 1+ \frac{1}{i+1} + \cdots + \frac{1}{n}$. Then we have $\frac{h_n}{\log n} \to 1$ as $n \to \infty$. Using these results, we can prove that $ \frac{c_n}{h_n} \to m \kappa^{-1}$ as $n \to \infty$. We observe that this implies $\frac{a_n}{m \log n} \to 1$ as $n \to \infty$.

Next we consider the variance. We will establish that for any $j \geq 1$ and $k=0 \cdots, m-1$, 
$$\Var[d_j(n+1, k+1)] \leq \Var[d_j(n+1,k)] + \frac{1}{n-1} $$
First we observe that
\[
E[d_j(n+1,k+1) | \mathscr{F}_{n+1,k}] 
= \left[ 1- \frac{1}{(n-1)(k + m (2n-1))}\right] d_j(n+1,k) + \frac{1}{n-1},
\]
and hence we get
\[
\Var[E[d_j(n+1,k+1) | \mathscr{F}_{n+1,k}]]
=  \left[ 1- \frac{1}{(n-1)(k + 2m (n-1))}\right]^2 \Var[d_j(n+1,k)].  
\]
Also, we have,
\begin{align*}
&\, \Var[d_j(n+1,k+1) | \mathscr{F}_{n+1,k}]
= \Var[d_j(n+1, k+1) - d_j(n+1,k) | \mathscr{F}_{n+1, k}] \\
&=\frac{1}{n-1}\left[ 1- \frac{d_j(n+1, k)}{ k + m(2n-1)}\right] \left[ 1 -  \frac{1}{n-1} \left[ 1- \frac{d_j(n+1, k)}{ k + m(2n-1)}\right]\right]\\
&\leq \frac{1}{n-1}\\
\end{align*}
Combining, we get, 
\begin{align*}
\Var[d_j(n+1,k+1)] &= \Var[E[d_j(n+1,k+1) | \mathscr{F}_{n+1,k}]] \\
&\ \qquad + E[\Var[d_j(n+1,k+1) | \mathscr{F}_{n+1,k}]] \\
&\leq \Var[d_j(n+1,k)] + \frac{1}{n-1}
\end{align*}
Using this relation repetitively, we get,
\begin{align*}
\Var[d_j(n+1)] &\leq \Var[d_j(n+1,m-1)] + \frac{1}{n-1}\\
&\leq \cdots \\
&\leq \Var[d_j(n)] + \frac{m}{n-1}
\end{align*}

Therefore, $\Var(d_j(n)) \leq C_0 \log n$ for some fixed constant $C_0$. Let $c <1$. For $n$ sufficiently large, $E[d_i(n)] \geq cm\log(n)$. We fix $\delta > 0$. Therefore, by Chebyshev inequality, for $n$ sufficiently large,
\begin{align*}
P\left[| \frac{d_j(n)}{E[d_j(n)]} -1 | > \delta\right] &\leq \frac{\Var(d_j(n))}{\delta^2 (E[d_j(n)])^2} \\
&\leq \frac{C_0 \log n} {\delta^2 \{c(m\log n)\}^2} \\
&= \frac{C_0}{\delta^2 c^2 m^2} \frac{1}{\log n} \\
&\to 0 \\ 
\end{align*}

This establishes that $\frac{d_j(n)}{E[d_j(n)]}  {\buildrel P \over \rightarrow} 1$ as $n\to \infty$. Combining this with the information that $\frac{E[d_j(n)]}{m\log n} \to 1$ as $n\to \infty$, we get that the degree of a fixed vertex scales like $m\log n$.

\subsection{Proof of Theorem \ref{ModelA_m_clt}}
The Central Limit Theorem will be derived by an application of the Martingale Central Limit Theorem. We need  to define a ``linear" sequence of random variables from the doubly indexed sequence $\{d_i(n,k)\}$ to apply the Central Limit Theorem. To this end, we define,
\begin{align*}
Z_{(k-1)m+j} &= d_i(i+k,j)
\end{align*}
Also, if $\{ \mathscr{F}_{n,k} \}$ denotes the natural filtration of the random graph process, then
we define 
\begin{align*}
\tilde{\mathscr{F}}_{(n-i-1)m+k} &= \mathscr{F}_{n,k} 
\end{align*}
We note that for $0\leq k \leq (n-i)m$, $Z_k = d_i(i+p,q)$, where $p = \floor*{\frac{k}{m}}+1$ and $q= k\pmod{m}$.

It may be easily observed that $P(Z_k- Z_{k-1}=1 | \tilde{\mathscr{F}}_{k-1}) = \zeta_k = 1- P(Z_k- Z_{k-1}=0 | \tilde{\mathscr{F}}_{k-1})$ where 
\begin{align*}
\zeta_k  &= \frac{1}{i+\floor*{\frac{k-1}{m}}-1} \left( 1- \frac{Z_{k-1}}{(k-1)\pmod{m} + m (2i +2 \floor*{\frac{k-1}{m}} -1)}\right)
\end{align*}
 We define the triangular array $\{ Y_{n,k} : k=1,\cdots, (n-i)m\}$,
\begin{align}
Y_{nk}&= \frac{Z_k - Z_{k-1} - E[Z_k - Z_{k-1} | \tilde{\mathscr{F}}_{k-1}]}{\sqrt{m\log n}} \nonumber
\end{align}
 We define $\sigma_{nk}^2 = E[Y_{nk}^2| \tilde{\mathscr{F}}_{k-1}] 
= \frac{1}{m\log n} \zeta_k(1-\zeta_k) $

Hence we have, $\displaystyle\sum_{k=1}^{(n-i)m} \sigma_{nk}^2 = I_1 - I_2$
where 
\begin{align}
I_1 &=  \frac{1}{m\log n} \sum_{k=1} ^{(n-i)m}\zeta_k \nonumber \\
I_2 &=  \frac{1}{m\log n} \sum_{k=1}^{(n-i)m} \zeta_k^2 \nonumber
\end{align}
Now, it follows that $I_2 {\buildrel P \over \rightarrow}  0$ as $n \to \infty$. 
\begin{align}
I_1 &= I_{11} - I_{12} \nonumber \\
I_{11} &= \frac{1}{m\log n}  \sum_{k =1}^{(n-i)m} \left[\frac{1}{i+ \floor*{\frac{k-1}{m}}-1} \right] \nonumber \\
I_{12} &= \frac{1}{m\log n} \sum_{k= 1}^{(n-i)m} \left[\frac{Z_{k-1}}{(i+\floor*{\frac{k-1}{m}}-1)((k-1)\pmod{m} + m (2i +2 \floor*{\frac{k-1}{m}} -1))} \right]\nonumber
\end{align}
$I_{11} \to 1$ as $ n \to \infty$. We will show that $I_{12} {\buildrel P \over \rightarrow} 0$. This will establish that $\displaystyle\sum_{k=1}^{(n-i)m} \sigma_{nk}^2{\buildrel P \over \rightarrow} 1$ as $n \to \infty$. We have seen in the proof of Theorem \ref{modelA_m} that $\frac{E[d_i(n)]}{m\log n} \to 1$ as $n \to \infty$. Also, $Z_k= d_i(i+1+\floor*{\frac{k-1}{m}}, k \pmod{m}) \leq d_i(i+1+\floor*{\frac{k-1}{m}})$.  We fix $\epsilon > 0$. Then $\exists N \geq i$ such that $\forall n \geq N$, $ \frac{E[d_i(n)]}{m\log n} \leq (1+ \epsilon) $. Therefore,
\begin{align}
E[I_{12}] &=  \frac{1}{m\log n} \sum_{k= 1}^{(n-i)m} \left[\frac{E[Z_{k-1}]}{(i+\floor*{\frac{k-1}{m}}-1)((k-1)\pmod{m} + m (2i +2 \floor*{\frac{k-1}{m}} -1))} \right]\nonumber  \\
&\leq \frac{1}{m\log n} \sum_{k= 1}^{(n-i)m} \left[\frac{E[Z_{k-1}]}{(i+\floor*{\frac{k-1}{m}}-1)(m (2i +2 \floor*{\frac{k-1}{m}} -1))} \right]\nonumber  \\
&\leq \frac{1}{m\log n} \sum_{k= 1}^{(n-i)m} \left[\frac{E[d_i(i+1+\floor*{\frac{k-1}{m}})]}{(i+\floor*{\frac{k-1}{m}}-1)(m (2i +2 \floor*{\frac{k-1}{m}} -1))} \right]\nonumber  \\
&= \frac{1}{m\log n} \sum_{j= 0}^{n-i-1} \left[\frac{E[d_i(i+j+1)]}{(i+j-1) (2i +2j -1)} \right]\nonumber  \\
&\leq   \frac{1}{m\log n} \sum_{j= 0}^{N+1} \frac{E[d_i(i+j-1)]}{(i+j-1)(2i+2j-1)}+  \frac{1}{\log n} \sum_{j= N+2}^{n-i-1} \frac{(1+\epsilon) \log(i+j-1)}{(i+j-1)(2i+2j-1)} \nonumber \\
&\to 0 \nonumber
\end{align}
An application of Markov inequality will establish that $I_{12}  {\buildrel P \over \rightarrow} 0$. Next, we note that $|Y_{nk}| \leq \frac{2}{\sqrt{m\log n}}$. Hence using the Martingale Central Limit Theorem, we have,
\begin{align}
\sum_{k=i+1}^{n} Y_{nk} &\cd N(0,1) \label{central_1} 
\end{align}
Finally, we observe that $ E[Z_k - Z_{k-1} | \tilde{\mathscr{F}}_{k-1}] = \zeta_k$. Hence, using similar arguments as above, we may conclude that
\begin{align}
\sum_{k=1}^{(n-i)m}\frac{ E[Z_k - Z_{k-1} | \tilde{\mathscr{F}}_{k-1}]- m\log n}{\sqrt{m\log n}} &{\buildrel P \over \rightarrow} 0 .\label{central_2}
\end{align}
Combining (\ref{central_1}) and (\ref{central_2}) completes the proof.

\subsection{Proof of Theorem \ref{ModelA_degree}}

We begin by observing a few important facts 
about the frequencies of degrees for the $m=1$ case.

\begin{proposition}
\label{ModelA_concentration}
Let $\left(G_n\right)_{n=1}^{\infty}$ be a sequence of 
random graphs following a 
\emph{linear de-preferential attachment} model with
$m=1$. 
Let 
\begin{equation}
N_k(n) := \sum_{i=1}^n \bone\left(d_i(n) = k\right),
\end{equation}
be the number of vertices with degree $k$. Then
$\exists\, C > 0$, such that, as $n \to \infty$.
\begin{equation}
P(\displaystyle\sup_k | N_k(n) - E[N_k(n)] | \geq C\sqrt{n\log n}) \to 0.
\end{equation}
\end{proposition}

\begin{proof}
We first observe that $N_k(n)= 0 $ if $k > n$. 
\begin{align*}
\ 
&    P(\displaystyle\sup_k |N_k(n) - E[N_k(n)] | \geq C \sqrt{n \log n} ) \\
&= P(\displaystyle\max_{k \leq n} |N_k(n) - E[N_k(n)] | \geq C \sqrt{n \log n} )\\
&\leq \displaystyle\sum_{k=1}^{n} P(|N_k(n) - E[N_k(n)] | \geq C \sqrt{n \log n} )
\end{align*}

We will prove that uniformly in $k\leq n$, 
\[
P(|N_k(n) - E[N_k(n)] | \geq C \sqrt{n \log n} )= o(n^{-1}).
\]

For $l=1.2, \cdots, n$, we define $M_l = E[N_k(n) | \mathscr{F}_l]$. $\{M_l\}_{l=1}^{n}$ is a martingale. $M_1= E[N_k(n) | \mathscr{F}_1] =E[N_k(n)]$ as $\mathscr{F}_1$ is the trivial sigma field. Also, $M_n = N_k(n)$. Therefore, $M_n- M_1 = N_k(n) - E[N_k(n)]$.  As the degree of at most 2 vertices is affected due to the addition of the $l^{th}$ vertex, we have $|M_l - M_{l-1}| \leq 2$. Then, by the Azuma-Hoeffding inequality, 
\begin{align*}
P(|N_k(n) - E[N_k(n)] | \geq a) &\leq 2 \exp\left[ -\frac{a^2}{ 8n}\right]  \\
\end{align*}

Taking $a= C\sqrt{n \log n}$ for any $C > 2\sqrt{2}$, we have, 
\begin{align*}
P(|N_k(n) - E[N_k(n)] | \geq C\sqrt{n \log n} ) &\leq 2 \exp\left[ -\frac{C^2 \log n }{ 8}\right] \\
&= o(n^{-1})\\
\end{align*}

This completes the proof of this proposition. 
\end{proof}

Our next proposition identifies the limit of the empirical degree distribution. 

\begin{proposition}
\label{ModelA_expected}
In the setup of the previous result, that is, 
Proposition ~\ref{ModelA_concentration},
$\exists$ a constant $C_1$ such that $\forall n \geq 2$ and all $k \in \mathbb{N}$,
 $$| E[N_k(n)] - p_k n | \leq C_1 \log n$$ where $p_k = (\frac{1}{2})^k$, $k\geq 1$. 
\end{proposition}

\begin{proof}
We first note that $p_k$ is the unique solution of the recurrence relation 
\begin{align*}
2p_k &= p_{k-1} + 1 _{(k=1)} \\
\end{align*}

where we define $p_0 = 0$. 

Recall, $N_k(n)$ denotes that number of vertices with degree exactly $k$ in the graph $G_n$, thus,
\footnotesize
\begin{equation}
N_K(n+1) - N_k(n) = \left\{
\begin{array}{rl}
+1 & \mbox{if\ \ } k=1 \mbox{\ or\ } \\
   & k > 1 
\mbox{\ and\ } (n+1)\mbox{-th vertex joins to a\ }(k-1) \mbox{\ degree vertex}\\
-1 &  \mbox{if\ \ }  (n+1)\mbox{-th vertex joins to a\ }k \mbox{\ degree vertex}\\
0 & \mbox{otherwise.} \\
\end{array}
\right.
\end{equation}
\normalsize
Thus,
\begin{align*}
&\, E[ N_k(n+1) - N_k (n) | \mathscr{F}_n]\\
&= \frac{1}{n-1} \left( 1- \frac{k-1}{2n-1} \right) N_{k-1} (n) - \frac{1}{n-1} \left(1- \frac{k}{2n-1}\right)N_k(n) + 1_{(k=1)}
\end{align*}
We define $\epsilon_k(n) = E[N_k(n)] - np_k$. We wish to prove that $\displaystyle\max_k |\epsilon_k(n)| \leq C_1 \log n $ for some fixed constant $C_1 > 0$. We first note that 
\begin{align*}
&\, E[N_k(n+1)] \\
&= \left(1- \frac{1}{n-1}\left( 1- \frac{k}{2n-1}\right)\right) E[N_k(n)] + \frac{1}{n-1} \left(1- \frac{k-1}{2n-1}\right) E[N_{k-1}(n)] + 1_{(k=1)}
\end{align*}
Also,
\begin{align*}
(n+1)p_k &= np_k + p_k \\
&= np_k + \frac{1}{2} p_{k-1} +\frac{1}{2} 1_{(k=1)} \\
&= \left(1- \frac{1}{n-1}\left(1- \frac{k}{2n-1} \right) \right)np_k + \frac{np_{k-1}}{n-1}\left(1- \frac{k-1}{2n-1}\right) \\
&+ \frac{np_k}{n-1} \left(1- \frac{k}{2n-1} \right) + \left( \frac{1}{2} - \frac{n}{n-1} \left(1- \frac{k-1}{2n-1} \right)\right) p_{k-1} + \frac{1}{2}1_{(k=1)} 
\end{align*}
Combining, we get, 
\begin{align*}
\epsilon_k(n+1) &=  \left(1- \frac{1}{n-1}\left(1- \frac{k}{2n-1} \right) \right) \epsilon_k(n) +  \frac{1}{n-1} \left(1- \frac{k-1}{2n-1}\right) \epsilon_{k-1}(n) \\
&+ \frac{1}{2}1_{(k=1)} - \gamma_k(n), 
\end{align*}
where $\gamma_k(n)= \frac{np_k}{n-1} \left(1- \frac{k}{2n-1} \right) + \left( \frac{1}{2} - \frac{n}{n-1} \left(1- \frac{k-1}{2n-1}\right)\right) p_{k-1} $. We will now use induction on $n$ to complete the proof. \\
We start with the base case $n=2$. In this case, $N_k(2) = 0$ for $ k>2$. Now, $|N_1(2) - 2p_1| = 0$ and $|N_2(2)- 2p_2|= \frac{1}{2}$. Therefore, 
$\displaystyle\max_k |\epsilon_k(2)| < C_1 \log 2$ holds whenever $C_1 > \frac{1}{\log 2}$. So the proposition holds for the base case. We will now assume that the proposition is true for $n$ and extend the proposition to $n+1$. So we assume that we have found a $C_1 >0 $ such that $ \displaystyle\max_k |\epsilon_k(n)| \leq C_1 \log n$. We first extend the proposition for $k=1$. We have,
\begin{align*}
\epsilon_1(n+1) &=  \left(1- \frac{1}{n-1}\left(1- \frac{1}{2n-1} \right) \right) \epsilon_1(n) + \frac{1}{2} - \frac{n}{2(n-1)} \left(1- \frac{1}{2n-1}\right) \\
&=  \left(1- \frac{1}{n-1}\left(1- \frac{1}{2n-1} \right) \right) \epsilon_1(n) - \frac{1}{2(n-1)} + \frac{n}{2(n-1)(2n-1)}, \\
\end{align*}
and thus we get 
\begin{align*}
|\epsilon_1(n+1)| &\leq  \left(1- \frac{1}{n-1}\left(1- \frac{1}{2n-1} \right) \right) |\epsilon_1(n)|  + \frac{1}{2(n-1)} + \frac{n}{2(n-1)(2n-1)}\\
&\leq |\epsilon_1(n)|  + \frac{1}{2(n-1)} + \frac{n}{2(n-1)(2n-1)} \\
&\leq C_1 \log n + \frac{C_1}{n+1} \leq C_1 \log(n+1).
\end{align*}

The inequalities are true whenever $C_1$ has been chosen so that $\frac{C_1}{n+1} \geq \frac{1}{2(n-1)} + \frac{n}{2(n-1)(2n-1)} $ for all $n \geq 2$. We observe that choosing $C_1 > 2$ suffices, as $ \frac{C_1}{n+1} - \frac{1}{2(n-1)} - \frac{n}{2(n-1)(2n-1)} =  \frac{C_1}{n+1} - \frac{3n-1}{2(n-1)(2n-1)} \geq \frac{2}{(n+1)} - \frac{2}{2n-1}= \frac{2(n-2)}{(n+1)(2n-1)} >0$. We have also used that $\log(n+1) \geq \log n + \frac{1}{n+1}$. This follows from the observation that 
\begin{align*}
\log(n+1) &= \int_{1}^{n+1} \frac{1}{y} dy \\
&= \int_{1}^{n} \frac{1}{y} dy + \int_{n}^{n+1} \frac{1}{y}dy \\
&\geq \log n + \frac{1}{n+1} 
\end{align*}

This extends the proposition for $k=1$. We now consider the case where $k >1$. We first look at the function $\gamma_k(n)$. We have,
\begin{align*}
|\gamma_k(n)| &= | \frac{np_k}{n-1} \left(1- \frac{k}{2n-1} \right) + \left( \frac{1}{2} - \frac{n}{n-1} \left(1- \frac{k-1}{2n-1} \right)\right) p_{k-1}| \\
&= | \frac{np_k}{n-1} - \frac{(n+1)p_{k-1}}{2(n-1)} + \frac{n(k-1)p_{k-1}}{(n-1)(2n-1)} - \frac{nkp_k}{(n-1)(2n-1)}| \\
&= | \frac{np_k}{n-1} - \frac{(n+1)p_k}{(n-1)} + \frac{n(k-1)p_{k-1}}{(n-1)(2n-1)} - \frac{nkp_k}{(n-1)(2n-1)}| \\
&= | -\frac{p_k}{n-1} + \frac{n}{(n-1)(2n-1)} \{ (k-1) p_{k-1} - kp_k \}|\\ 
&\leq \frac{1}{n-1} +  \frac{n}{(n-1)(2n-1)} \displaystyle\sup_k |(k-1)p_{k-1} - p_k|  \\
&\leq \frac{\Delta}{n+1} \\
\end{align*}

The last inequality follows from the observation that $|\{ (k-1) p_{k-1} - kp_k \}| = | kp_k - p_{k-1}| \leq 2$. Now, we have, 
\begin{align*}
&|\epsilon_k(n+1)| \\
&= \left(1- \frac{1}{n-1} \left( 1- \frac{k}{2n-1}\right)\right) |\epsilon_k(n)| + \frac{1}{n-1} \left(1- \frac{k-1}{2n-1}\right)|\epsilon_{k-1}(n)|+|\gamma_k(n)|  \\
&\leq  \left(1- \frac{1}{n-1} \left( 1- \frac{k}{2n-1}\right)\right)C_1 \log n + \frac{1}{n-1} \left(1- \frac{k-1}{2n-1}\right) C_1 \log n + \frac{\Delta}{n+1} \\
&\leq C_1 \log n + \frac{1}{(n-1)(2n-1)}C_1 \log n + \frac{\Delta}{n+1}  \\
&\leq C_1 \log n + \frac{C_1}{n+1} \\
\end{align*}

The inequalities hold as long as $C_1$ is chosen large enough so that $\frac{C_1 }{n+1} \geq  \frac{1}{(n-1)(2n-1)} C_1 \log n + \frac{\Delta}{n+1} $ for all $n \geq 2$. This holds as long as $C_1 \geq \frac{\Delta}{ 1- 
\frac{(n+1) \log n}{(n-1)(2n-1)}}$ for all $n \geq 2$. \\
So, we select the constant $C_1$ such that it satisfies the requirements outlined above. This completes the proof by induction. 
\end{proof}

 We assume, without loss of generality, that $C_1$ in Proposition \ref{ModelA_expected} is chosen to be larger than $C$ in Proposition \ref{ModelA_concentration}. 
We have established that $P (\displaystyle\max_k | N_k(n) - E[N_k(n)]| > C_1 \sqrt{n \log n}) =o(1)$. Also, $\displaystyle\max_k| E[N_k(n)] - p_k n | \leq C_1 \log n$. \\
\begin{align*}
P(\displaystyle\max_k |N_k(n) - np_k| > C_1 (\log n + \sqrt{n\log n}) &\leq\\
 P(\displaystyle\max_k | N_k(n) - E[N_k(n)]| > C_1 \sqrt{n \log n}) &= o(1)
\end{align*}

This implies that 
\begin{align*}
P\left( \displaystyle\max_k | P_k(n) - p_k | > C_1 \left( \frac{\log n}{n} + \sqrt{\frac{\log n}{n}}\right)\right) &= o(1).
\end{align*}
This completes the proof. 

\subsubsection{Proof of Theorem \ref{ModelA_sizebiased}}
We have, by Theorem \ref{ModelA_degree},
\begin{align*}
P(e_{n+1} = \{v_{n+1},v\} , d_v(n)=k|\mathscr{F}_n) &= N_k(n) \frac{1}{n-1}\left(1- \frac{k}{2n-1}\right) \\ 
&\cd \left(\frac{1}{2}\right)^k 
\end{align*}
We note that the sequence $N_k(n) \frac{1}{n-1}\left(1- \frac{k}{2n-1}\right)$ is uniformly bounded by $2$ and hence uniformly integrable. Therefore, by taking expectations, we have,
\begin{align*}
P(e_{n+1} = \{v_{n+1},v\} , d_v(n)=k) &\to \left(\frac{1}{2}\right)^k
\end{align*}

\section{Proofs for the Inverse De-Preferential Models}
\label{proofs-inv}
In this section we provide proofs of the main results for the inverse de-preferential case.

\subsection{Embedding}
\label{sec;embedding}
We begin by providing two very natural yet important embeddings for the inverse de-preferential models, which 
are namely, \emph{Crump-Mode-Jagers(CMJ) branching process} to be used for the $m=1$ case and 
a sequence of 
\emph{pure birth processes} to be used for the $m>1$ case. Such natural couplings will allow us to analyze various properties of the random graph sequence. The approach 
for the $m=1$ case will follow the analysis by Rudas, T\'{o}th and Valko(\cite{RudTotVal2007}). 

\subsubsection{CMJ Branching Process}
\label{onehalfedge}
$\mathcal{G}$ denotes the space of finite rooted trees. If $\mathscr{T} \in \mathcal{G}$ and $x$ is a vertex of $\mathscr{T}$, then we define $(\mathscr{T})_{\downarrow x}$ as the sub-tree consisting of the the descendants of $x$.  \\
We start with a graph consisting of a vertex with a half edge. Each vertex reproduces independently according to i.i.d copies of a pure birth process $\{ \xi(t): t \geq 0\}$.  $\xi(0) = 1 $ w.p. 1 and 
\begin{align*}
P(\xi(t+h)= k+1 | \xi(t)=k) &= \frac{h}{k} + o(h)
\end{align*}

Let $\{\Upsilon(t): t\geq 0 \}$ denote the randomly growing tree. For every $t \geq 0$, $\Upsilon(t) \in \mathcal{G}$. We define the following sequence of random times 
\begin{align*}
\tau_2 & := \inf\{t \geq 0: |\Upsilon(t)|=2\} \\
\tau_3 &:= \inf\{t \geq \tau_2: |\Upsilon(t)| = 3\} \\
&\vdots\\
\tau_n&:= \inf\{t \geq \tau_{n-1} : |\Upsilon(t)|= n \} \\
&\vdots
\end{align*}
We look at the tree $\Upsilon(t)$ at the random times $\{\tau_n\}$. Let $\{G_n\}_{n=2}^{\infty}$ denote the random graph sequence under the inverse de-preferential attachment case when $m=1$. This gives us the following result. 
\begin{theorem}
The sequence of random graphs $\{G_n\}_{n=2}^{\infty}$ is distributed identically as $\{ \Upsilon(\tau_n)\}_{n=2}^{\infty}$.
\end{theorem}
We will find it useful to consider the pure birth process $\{\xi(t): t \geq 0\} $ as a point process, 
where the points occur at the birth times $\{T_n\}_{n=1}^{\infty}$ of the pure birth process. 
We define the expected Laplace Transform of the point process $\{\xi(t): t \geq 0\} $
\begin{align*}
\hat{\rho} (\lambda) &= E \left(\displaystyle\int_{0}^{\infty} \exp(- \lambda t) d\xi(t)\right)\\
& = \displaystyle\sum_{n=1}^{\infty}\prod_{i=0}^{n-1} \frac{1}{(i+1)\lambda + 1}
\end{align*}
$\hat{\rho}(\lambda)$ may be calculated easily because $\{T_n - T_{n-1}\}$ are independent random variables with $Exp(1/n)$ distribution, that is, exponential distributions with mean $n$.
We observe that the equation $\hat{\rho}(\lambda) =1$ has a unique root $\lambda= \lambda^*$. $\lambda^*$ is usually referred to as the Malthusian 
Parameter in the context of Crump-Mode-Jagers Branching Processes. The process $\{\Upsilon(t):t\geq 0\}$ is a supercritical, Malthusian Branching Process. 
Then using a theorem from O.Nerman (1981)(Theorem A, \cite{RudTotVal2007}), we have the following result. 
\begin{theorem}
\label{embedding}
Consider a bounded function $\phi : \mathcal{G} \to \mathbb{R}$. Then the following limit holds almost surely
\begin{align*}
\lim_{t \to \infty} \frac{1}{|\Upsilon(t)|} \displaystyle\sum_{x \in \Upsilon(t)}\phi(\Upsilon(t)_{\downarrow x}) &= \lambda^{*} \int_{0}^{\infty} \exp\{-\lambda^{*} t\} \mathbf{E}(\phi(\Upsilon(t)))dt.
\end{align*}
\end{theorem}

\subsubsection{Athreya-Karlin Embedding}
\label{m half edges}
For $m > 1$, we will couple our graph process with 
a sequence of Yule processes (i.i.d pure birth processes), with appropriate birth rates, such that, 
our degree sequences at each vertex will have the same
distribution as the number of particles in the respective Yule processes sampled at suitable random times. 
Similar coupling has been used in \cite{Ath2007, AthGhoSeth2008} and also in \cite{BuDa90}, 
where it is termed as 
\emph{Rubin’s construction} in the context of reinforced random walks. The coupling with our specific birthrates,
namely, $\lambda_i = \sfrac{1}{i}$
has also appeared in a recent work of Thacker and Volkov \cite{ThVol2018}.

It is worth mentioning here that unlike in the previous case $m=1$, when the entire graph process can be embedded inside a branching tree, in this case we are not embedding the entire graph process into the random forest which can be obtained from the Yule processes. Instead, this coupling is essentially only for the degree sequences at each vertex. Thus, having \emph{self-loops} and \emph{multiple edges} in our
random graphs, which can occur since $m > 1$, do not create any contradiction to this coupling.


Let $\{Z(t): t \geq 0\}$ be a pure birth process with $Z(0) = m$ w.p. 1 and birth rates $\{ \lambda_i \}_{i=m}^{\infty}$, $\lambda_i = \frac{1}{i}$.  Let $\{Z_i (t): t \geq 0\}_{i=1}^{\infty}$ be i.i.d. copies of the pure birth process $Z(t)$. 

We will define a sequence of random times $\{ \tau_n \}_{n=1}^{\infty} $. Let $\tau_1 = 0$ w.p.1. We start the process $Z_1(t)$ at $t=0$. Let $T_1^{(2)}$ be the time after $\tau_1$ when the first birth takes place in $Z_1$. Let $T_2^{(2)}$ be the time after $\tau_1 + T_1^{(2)}$ when the second birth occurs in $Z_1$. We continue in this manner to get $T_1^{(2)}, \cdots, T_m^{(2)}$. Let $\tau_2$ be the time when the $m^{th}$ birth occurs in ${Z}_1(t-\tau_1)$.  We start a new process ${Z_2}(t)$ at $t= \tau_2$. We have, 
\begin{align*}
\tau_2- \tau_1 &= T_1^{(2)} + \cdots + T_m^{(2)} 
\end{align*}
In general, let $T_k^{(n+1)}$ denote the time of $k^{th}$ cumulative birth in $Z_1,\cdots,Z_n$ after $\tau_n+ T_1^{(n+1)}+\cdots+T_{k-1}^{(n+1)}$. Let $\tau_{n+1}$ be the time after $\tau_n$ when the $m^{th}$ birth takes place in the processes ${Z}_1,{Z}_2, \cdots, {Z}_n$ after $\tau_n$. We start the process $Z_{n+1}(t)$ at $t= \tau_{n+1}$.Therefore, we have,
$$\tau_{n+1} - \tau_n = T_1^{(n+1)} + \cdots + T_m^{(n+1)}$$
 We define, for $j = 1,\cdots,n$, 
\begin{align*}
\tilde{d}_j(n+1,k) &= {Z}_j (\tau_n + T_1^{(n+1)} + \cdots + T_k^{(n+1)} - \tau_i)
\end{align*}
We have the following embedding result. 
\begin{theorem}
\label{embedding_m}
The sequence $\{\tilde{d}_j(n+1,k), k= 0, \cdots, m-1,j=1,\cdots,n , n \geq 2 \}$ and $\{ d_j (n+1,k), k=0, \cdots, m-1, j=1, \cdots,n, n \geq 2 \} $ are identically distributed.
\end{theorem}
It is important to emphasize at this point that the embedding of the random graph process in the $m=1$ case into a CMJ branching process induces an Athreya-Karlin Embedding of the random graph process in the same probability space. We will be utilizing both these embeddings for the $m=1$ case to establish certain properties of these random graphs. 

We note that the pure birth process $Z(t)$ considered in Section \ref{m half edges} reduces to the birth process $\xi(t)$ considered in Section \ref{onehalfedge} if we fix $m=1$. We will need the following results about the asymptotic properties of these pure birth processes. 

\begin{theorem}
\label{purebirth_rate}
Let $\{ Z(t) :  t \geq 0 \} $ be a pure birth process with $Z(0) = m$ w.p. 1 and birth rates $\{ \lambda_i \}_{ i=m}^{\infty}$, $\lambda_i = \frac{1}{i}$.  Then $\frac{Z(t)}{\sqrt{t}}  {\buildrel P \over \rightarrow}  \sqrt{2}$ as $t \to \infty$. 
\end{theorem}

\begin{proof}
Let 
\begin{align*}
T_1 &= 0 \\
T_2 &= \inf\{t\geq T_1: Z(t) = m+1\} \\
T_3 &= \inf\{t \geq T_2: Z(t)= m+2\} \\
&\vdots
\end{align*}
Then we know, $T_1, T_2 - T_1, \cdots, T_n- T_{n-1}, \cdots$ are independent exponential random variables. Therefore, $L_n= T_{n+1}- T_n$ are independent and $L_n \sim Exp(\frac{1}{m+n-1})$.Then we have
\begin{align*}
T_n&= L_1 + \cdots + L_{n-1} \\
\implies E(T_n) &= m + (m+1) + \cdots + (m+ n-2) \\ 
\implies E(T_n) &= (n-1)m + \frac{(n-1)(n-2)}{2} 
\end{align*}
Also, $\Var(L_n)=(m+n-1)^2$. Hence we have,
\begin{align*}
\sum_{k=1}^{\infty} \Var\left(\frac{L_k}{k^2}\right) &= \sum_{k=1}^{\infty} \frac{(m+k-1)^2} {k^4} < \infty.
\end{align*}
Thus, $\sum_{k=1}^{\infty} \frac{L_k - E[L_k]}{k^2} <\infty$ w.p. 1. 
First, this implies that $\frac{L_k}{k^2}\to  0$ a.s., because $E(L_k)=(m+k-1)$. Further,
an application of Kronecker's Lemma yields that $\frac{T_n - E[T_n]}{n^2} \to 0$ w.p. 1. This allows us to conclude that $\frac{T_n}{n^2} \to \frac{1}{2}$ a.s. 

We observe that $Z(t) \uparrow \infty$ as $t\to \infty$ and $T_{Z(t)} \leq t < T_{Z(t)+1}$. Therefore,
\begin{align*}
\frac{T_{Z(t)}}{(Z(t))^2} \leq &\frac{t}{(Z(t))^2} < \frac{T_{Z(t)}}{(Z(t))^2} + \frac{L_{Z(t)}}{(Z(t))^2}\\
\implies \frac{Z(t)}{\sqrt{t}} &\to \sqrt{2} \quad a.s.
\end{align*}
This completes the proof. 
\end{proof}

\subsection{Technical Results on Normalizing Constant}
The results established below allow us to approximate the normalizing constants for the 
de-preferential attachment model. 

Let $\tilde{\mathscr{F}}_{n,k}$ denote the natural filtration associated with the continuous time embedding $\tilde{d}_j(n+1,k)$ described in section \ref{m half edges}. We define 
$$ \tilde{C}_{n+1, k}^{-1} = \tilde{D}_{n+1,k} = \sum_{j=1}^{n} \frac{1}{\tilde{d}_j(n+1,k)}.$$
The natural filtration $\{\tilde{\mathscr{F}}_n\}$ and the constants $\tilde{C}_n$ and $\tilde{D}_n$ are defined analogously for the $m=1$ case. 

\begin{lemma}
\label{ModelB_bound}
For $m=1$, $\frac{\tilde{D}_n}{n} \to \lambda^*$ a.s.
\end{lemma}
\begin{proof}
The proof follows from Theorem A, \cite{RudTotVal2007} by using the bounded functional 
$\phi(G)= \frac{1}{1+ \#\text{children}(\emptyset)}$, where $\emptyset$ is the root of a finite tree $G \in \GG$.
\end{proof}

For $m>1$, we have the following bounds on the normalizing constants. 
\begin{lemma}
\label{ModelB_m_bound}
$\forall n \geq 2, \forall k=0,1,\cdots, m-1$, $ \frac{m}{n-1} \leq C_{n,k} \leq \frac{2m}{n-1}$ wp 1.
\end{lemma}
\begin{proof}
We note that $C_{n,k}^{-1}=D_{n,k}= \displaystyle\sum_{j=1}^{n-1} \frac{1}{d_j(n,k)}$. We observe that $d_j(n,k) \geq m$ for $j=1, \cdots, n-1$ and therefore $D_{n,k} \leq \frac{n-1}{m}$. We also observe that $\displaystyle\sum_{j=1}^{n-1} d_j(n,k) = k + m(2n-3) $. Therefore, by the A.M.-H.M. inequality,
\begin{align*}
\displaystyle\sum_{j=1}^{n-1} \frac{1}{d_j(n,k)} &\geq \frac{(n-1)^2}{k+ m(2n-3)} \\
&\geq \frac{n-1}{2m}
\end{align*}
Combining, we get that $\frac{n-1}{2m} \leq D_{n,k} \leq \frac{n-1}{m}$. Therefore, we have, $\frac{m}{n-1} \leq C_{n,k} \leq \frac{2m}{n-1}$ w.p. 1. 
\end{proof}
We use Theorem \ref{embedding_m} to conclude that $\tilde{D}_{n+1,k}$ and $D_{n+1,k}$ are identically distributed. Therefore, using the previous lemma, we have,
\begin{equation}
\label{embedding_bound}
 \frac{m}{n} \leq \tilde{C}_{n+1,k} \leq \frac{2m}{n}
\end{equation}

\begin{proposition}
\label{modelB_m_rate}
$\forall i \geq 1$, $\exists$ a random sequence $\{c_n\} \sim \Theta( m^2 \log n )$ such that 
$$\frac{\tau_n - \tau_ i}{ c_n} \to 1 \quad a.s.$$ as $n \to \infty$.  
\end{proposition}

\begin{proof}
We observe that the random variables $T_1^{(n+1)}, \cdots, T_m^{(n+1)}$ are independent and that 
\begin{align*}
T_1^{(n+1)} | \tilde{\mathscr{F}}_{n+1,0} &\sim Exp(\tilde{D}_{n+1,0}) \\
T_2^{(n+1)} | \tilde{\mathscr{F}}_{n+1,1} &\sim Exp(\tilde{D}_{n+1,1}) \\
&\vdots\\
T_m^{(n+1)} | \tilde{\mathscr{F}}_{n+1,m-1} &\sim Exp(\tilde{D}_{n+1,m-1}) 
\end{align*} 
We define 
\begin{align*}
b_n &= \tilde{C}_{n+1,0} +E[\tilde{C}_{n+1, 1}| \tilde{\mathscr{F}}_{n+1,0}]+ \cdots + E[\tilde{C}_{n+1, m-1}| \tilde{\mathscr{F}}_{n+1,0}]
\end{align*}
Then, using equation \eqref{embedding_bound}, we may conclude that,
\begin{align}
\label{Equ:bound-on-b-n}
\frac{m^2}{n} \leq &b_n \leq \frac{2m^2}{n}
\end{align}
Now, recall that 
$\tau_{n+1} - \tau_n = \displaystyle{\mathop{\sum}\limits_{k=1}^{m} T_k^{(n+1)}}$. Thus,
$\{ \tau_{n+1} - \tau_n - b_n , \tilde{\mathscr{F}}_{n+1,0} \}$ forms a martingale difference sequence. Further, we have, $\sum b_n^2 < \infty$, which implies that $Y_n = \tau_n - \tau_i - (b_i +\cdots+ b_n)$ is an $L^2$ bounded martingale. Therefore, by the Martingale Convergence Theorem, we have, $Y_n \to Y$ a.s. as $n\to \infty$. We define,
\begin{align*}
c_n &= b_i + \cdots + b_n
\end{align*}
Therefore,
\begin{align*}
\tau_n - \tau_i - c_n & {\buildrel a.s. \over \rightarrow} Y\\
\implies \frac{\tau_n - \tau_i}{c_n} & {\buildrel a.s. \over \rightarrow} 1 
\end{align*}
Again, using equation 
\eqref{Equ:bound-on-b-n}, it easily follows that $c_n \sim \Theta(m^2 \log n)$. This completes the proof. 
\end{proof}

\begin{proposition}
\label{Model_exact1}
For $m=1$, the sequence $\{c_n\}$ outlined in Proposition \ref{modelB_m_rate} satisfies $\frac{c_n}{\log n} \to \frac{1}{\lambda^*}$ as $n \to \infty$. 
\end{proposition}

\begin{proof}
From Lemma \ref{ModelB_bound}, we have $n\tilde{C}_n\to \frac{1}{\lambda^*}$ a.s. as $n \to \infty$. This observation, along with the
form of the sequence $c_n$, help us to conclude the result sought. 
\end{proof}

\subsection{Proof of Theorem \ref{ModelB_exact}}
From Theorem \ref{embedding_m} we have, for $j=1, \cdots, n$, $d_j(n)$ is distributed identically as $Z_j(\tau_n - \tau_j)$. Also, combining Theorem \ref{purebirth_rate}  and Proposition \ref{modelB_m_rate} ,we have,
\begin{align*}
\frac{Z_i(\tau_n- \tau_i)}{\sqrt{c_n}} \to  \sqrt{2} \quad a.s.\\
\implies \frac{d_i(n)}{\sqrt{c_n}}  \to \sqrt{2} \quad a.s.
\end{align*}
Finally, we note from Proposition \ref{Model_exact1}, that $\frac{c_n}{\log n} \to \frac{1}{\lambda^*}$ a.s. This helps us to conclude that the result in consideration. 

\subsection{Proof of Theorem \ref{ModelB_degree} }
We use Theorem A, \cite{RudTotVal2007} with $\phi(G) = 1( \#children(\emptyset, G) =k)$ where $\emptyset$ denotes the root of the tree $G$. Then we have,
\begin{align*}
&\displaystyle\lim_{t \to \infty} \frac{ | \{x\in \Upsilon(t) : deg(x,\Upsilon(t)) = k+1\}|}{|\Upsilon(t)|} \\
&= \lambda^* \displaystyle\int_{0}^{\infty} \exp(-\lambda^*t) P( \# children(\emptyset, \Upsilon(t))=k) dt 
\end{align*}
Now,
\begin{align*}
& P( \# children(\emptyset, \Upsilon(t))=k) = P(T_k <t) - P(T_{k+1} <t)\\
\end{align*}
By Fubini's Theorem, we have, 
\begin{align*}
\lambda^* \displaystyle\int_{0}^{\infty} \exp(-\lambda^*t) P(T_k <t) dt &= E(e^{-\lambda^* T_k})
\end{align*}
$T_k$ is the sum of independent exponentially distributed random variables with parameters $1,\frac{1}{2}, \frac{1}{3},\cdots, \frac{1}{k}$, this can be easily calculated. 
This completes the proof. 

\subsection{Proof of Corollary} 
\ref{ModelB_character}
We first define 
\begin{align*}
\tilde{p}_k &= \frac{k\lambda^*}{ k\lambda^* + 1} \displaystyle\prod_{i=1}^{k-1} \frac{1}{i\lambda^* + 1}, \quad k=1,2, \cdots
\end{align*}
\begin{itemize}
\item[(i)] We observe that $\lambda^*>1$ and note that $\frac{\tilde{p}_{k+1}}{\tilde{p}_k}\leq 1$ $\forall k\geq 1$ . This proves that the mode of the distribution is at $1$.
\item[(ii)] We observe that 
\begin{align*}
\displaystyle\sum_{k=n}^{\infty} \tilde{p}(k) &= \prod_{i=1}^{n-1}\frac{1}{1+i\lambda^*}\\
\implies \displaystyle\sum_{n=1}^{\infty}\displaystyle\sum_{k=n}^{\infty} \tilde{p}(k) &= 2. 
\end{align*}
\item[(iii)] 
\begin{align*}
\displaystyle\sum_{k=n}^{\infty} \tilde{p}(k) &= \prod_{i=1}^{n-1}\frac{1}{1+i\lambda^*}\\
&= \frac{1}{(\lambda^*)^{n-1}} \frac{\Gamma\left(1+ \frac{1}{\lambda^*}\right)}{ \Gamma\left(n  + \frac{1}{\lambda^*}\right)}
\end{align*}
The proof may be completed by applying Stirling's approximation for the Gamma function, 
$ \Gamma(x+1) \sim \displaystyle\sqrt{2 \pi} \exp(-x) x^{(x+ \frac{1}{2})}$.
\end{itemize}

\subsection{Proof of Theorem \ref{ModelB_sizebiased}}
We will use the Athreya-Karlin Embedding. Let $\tilde{N}_k(n)$ denote the number of 
processes with exactly $k$ individuals at time $\tau_n$. Hence, combining Lemma \ref{ModelB_bound} and Theorem \ref{ModelB_degree} , we have,
\begin{eqnarray*}
P(\tilde{d}_v(n+1) - \tilde{d}_v(n)=1,\tilde{d}_v(n)= k|\tilde{\mathscr{F}}_n ) &= \tilde{N}_k(n). \tilde{C}_n \frac{1}{k} \\
&= \frac{\tilde{N}_k(n)}{n}. n\tilde{C}_n \frac{1}{k} \\
&\cas \frac{1}{k\lambda^*} \frac{k\lambda^*}{ k\lambda^* + 1} \displaystyle\prod_{i=1}^{k-1} \frac{1}{i\lambda^* + 1} \\
&= \displaystyle\prod_{i=1}^{k} \frac{1}{1+ i\lambda^*} 
\end{eqnarray*}
We note that the sequence $\tilde{N}_k(n). \tilde{C}_n \frac{1}{k}$ is uniformly bounded and hence uniformly integrable. Taking expectations gives us the result sought.

\subsubsection{Proof of Theorem \ref{ModelB_m}}
From Theorem \ref{embedding_m} we have, for $j=1, \cdots, n$, $d_j(n)$ is distributed identically as $Z_j(\tau_n - \tau_j)$. Also, combining Theorem \ref{purebirth_rate}  and Proposition \ref{modelB_m_rate} ,we have,
\begin{align*}
\frac{Z_i(\tau_n- \tau_i)}{\sqrt{c_n}}  \to \sqrt{2} \quad a.s.\\
\implies \frac{d_i(n)}{\sqrt{c_n}} \to \sqrt{2} \quad a.s. 
\end{align*}
Finally, we note from Proposition \ref{modelB_m_rate}, that $c_n \sim \Theta(m^2 \log n)$. This concludes the proof. 


\section*{Acknowledgment}
The authors would like to thank the anonymous referee for 
her/his valuable remarks which have significantly improved the quality of presentation of the current version of the
work. Many thanks are also due to the editors for their
kind consideration and patience with us. Last but not least, both authors feel honored to be part of this special
series in memory of late Professor K. R. Parthasarathy, to
whom the entire Indian mathematics and statistics community
is in debt. While this work was being carried out Professor
K. R. Parthasarathy and the first author had many discussions
on this new type of random graphs, and his ever insightful
suggestions motivated both the authors immensely.

\bibliographystyle{plain}

\bibliography{De-Pref}

\begin{thebibliography}{10}

\bibitem{Ath2007}
K.~B. Athreya.
\newblock Preferential attachment random graphs with general weight function.
\newblock {\em Internet Math.}, 4(4):401--418, 2007.

\bibitem{AthGhoSeth2008}
Krishna~B. Athreya, Arka~P. Ghosh, and Sunder Sethuraman.
\newblock Growth of preferential attachment random graphs via continuous-time
  branching processes.
\newblock {\em Proc. Indian Acad. Sci. Math. Sci.}, 118(3):473--494, 2008.

\bibitem{BaKa2018}
Antar Bandyopadhyay and Gursharn Kaur.
\newblock Linear de-preferential urn models.
\newblock {\em Adv. in Appl. Probab.}, 50(4):1176--1192, 2018.

\bibitem{BarAlb99}
Albert-L{\'a}szl{\'o} Barab{\'a}si and R{\'e}ka Albert.
\newblock Emergence of scaling in random networks.
\newblock {\em Science}, 286(5439):509--512, 1999.

\bibitem{BuDa90}
Burgess Davis.
\newblock Reinforced random walk.
\newblock {\em Probab. Theory Related Fields}, 84(2):203--229, 1990.

\bibitem{JanLucRuc2000}
Svante Janson, Tomasz {\L}uczak, and Andrzej Rucinski.
\newblock {\em Random graphs}.
\newblock Wiley-Interscience Series in Discrete Mathematics and Optimization.
  Wiley-Interscience, New York, 2000.

\bibitem{Ka2019}
Gursharn Kaur.
\newblock Negatively reinforced balanced urn schemes.
\newblock {\em Adv. in Appl. Math.}, 105:48--82, 2019.

\bibitem{MirFig2009}
M.~N. Miranda and D.~R. Figueiredo.
\newblock A {P}referential {A}ttachment {M}odel for {T}ree {C}onstruction in
  {P}2{P} {V}ideo {S}treaming.
\newblock WPerformance (available at {\tt
  <http://www.lbd.dcc.ufmg.br/colecoes/wperfomance/2009/012.pdf>}), 2009.

\bibitem{RudTotVal2007}
Anna Rudas, B{\'a}lint T{\'o}th, and Benedek Valk{\'o}.
\newblock Random trees and general branching processes.
\newblock {\em Random Structures Algorithms}, 31(2):186--202, 2007.

\bibitem{SevRik2006}
V~Sevim and P.~A. Rikvold.
\newblock Effects of {P}reference for {A}ttachment to {L}ow-degree {N}odes on
  the {D}egree {D}istributions of a {G}rowing {D}irected {N}etwork and a
  {S}imple {F}ood-{W}eb {M}odel.
\newblock {\em Phys. Rev. E}, 73:5(056115):(7 pages), 2006.

\bibitem{SevRik2008}
V~Sevim and P.~A. Rikvold.
\newblock Network {G}rowth with {P}referential {A}ttachment for {H}igh
  {I}ndegree and {L}ow {O}utdegree.
\newblock {\em Physica A}, 387:2631--2636, 2008.

\bibitem{ThVol2018}
Debleena Thacker and Stanislav Volkov.
\newblock Border aggregation model.
\newblock {\em Ann. Appl. Probab.}, 28(3):1604--1633, 2018.

\bibitem{Hofs2012}
Remco van~der Hofstad.
\newblock Random {G}raphs and {C}omplex {N}etworks.
\newblock (unpublished, available at {\tt
  <http://www.win.tue.nl/~rhofstad/NotesRGCN.pdf>}), 2013.

\end{thebibliography}

\end{document}